\documentclass{amsart}
\usepackage{amsmath,amssymb,amsthm,amsfonts,epsfig,verbatim}
\usepackage[all]{xy}

\newcommand{\liealgebra}[1]{{\mathfrak {#1}}}
\newcommand{\End}{\liegroup{End}}
\newcommand{\gl}{\liealgebra{gl}}
\newcommand{\so}{\liealgebra{so}}
\newcommand{\lsl}{{\liealgebra{sl}}}

\newcommand{\un}{\liealgebra{u}}

\newcommand{\fa}{\liealgebra{a}}

\newcommand{\lk}{\liealgebra{k}}
\newcommand{\lp}{\liealgebra{p}}
\newcommand{\liegroup}[1]{{\operatorname{#1}}}
\newcommand{\G}{\liegroup{G}}

\newcommand{\GL}{\liegroup{GL}}
\newcommand{\SL}{\liegroup{SL}}
\newcommand{\SO}{\liegroup{SO}}

\newcommand{\Un}{\liegroup{U}}
\newcommand{\Or}{\liegroup{O}}

\newcommand{\im}{\operatorname{im}}
\newcommand{\eps}{\epsilon}

\newcommand{\mcl}{\mathcal L}

\newcommand{\cp}[1]{{\C \mathbb{P}^{#1}}}

\newcommand{\al}{\alpha}

\newcommand{\tr}{\rm tr}

\newcommand{\la}{\lambda}

\newcommand{\rlg}{{\mcl}}

\def\Id{\operatorname{Id}}

\renewcommand{\mathbb}{\mathsf}

\newcommand{\R}{\mathbb R}

\newcommand{\C}{\mathbb C}

\usepackage[mathscr]{eucal}
\usepackage{bbm}
\usepackage{oldgerm}

\theoremstyle{plain}
\newtheorem{thm}{Theorem}[section]
\newtheorem{lem}[thm]{Lemma}

\newtheorem{cor}[thm]{Corollary}

\newtheorem{prop}[thm]{Proposition}

\theoremstyle{definition}
\newtheorem{rem}[thm]{Remark}

\newtheorem{exam}[thm]{Example}

\subjclass{22E67, 37K25}

\begin{document}
\title[Noncompact Reality Conditions]{Generating Rational Loop Groups With Noncompact Reality Conditions}
\author{Oliver Goertsches}
\address{Universit\"at Hamburg, Fachbereich Mathematik, Bundesstra\ss e 55, 20146 Hamburg, Germany}
\email{oliver.goertsches@math.uni-hamburg.de}

\thanks{The work on this paper was started when the author was visiting Chuu-Lian Terng in Irvine, supported by a DAAD postdoctoral scholarship. He wishes to thank UC Irvine, and especially Chuu-Lian Terng, for their hospitality. Furthermore, he thanks Chuu-Lian Terng, Daniel Fox, and Neil Donaldson for valuable comments and useful discussions.}

\begin{abstract}
We find generators for the full rational loop group of $\GL(n,\C)$ as well as for the subgroup consisting of loops that satisfy the reality condition with respect to the noncompact real form $\GL(n,\R)$. We calculate the dressing action of some of those generators on the positive loop group, and apply this to the ZS--AKNS flows and the $n$--dimensional system associated to $\GL(n,\R)/\Or(n)$.
\end{abstract}

\maketitle

\section{Introduction}
The interest in finding generators for rational loop groups, i.e.~groups of meromorphic maps from $\cp{1}$ into a complex Lie group, originated from dressing actions \cite{Terng2000} and their various geometric applications; cf.~the survey \cite{Terng2008} and the references therein. Terng and Uhlenbeck introduced the idea of simple elements, i.e.~rational loops with as few poles as possible that generate the loop group, in order to obtain explicit formulae for the dressing action.

Uhlenbeck \cite{Uhlenbeck1989} found simple elements for the group of $\GL(n,\C)$-valued rational loops satisfying the $\Un(n)$-reality condition, and Terng and Wang \cite{Terng2006} extended this to the twisted loop group associated to $\Un(n)/\Or(n)$. Motivated by this work, Donaldson, Fox and the author \cite{DFG2008}
found generators for the rational loop groups of all classical groups and $\G_2$ with reality condition given by the respective compact real form, and most of their twisted loop groups.

Looking at the above results, it suggests itself to ask for generators of rational loop groups, where the reality condition is given by a \emph{noncompact} real form. In this paper, we solve this question for the easiest case, namely the noncompact real form $\GL(n,\R)$ of $\GL(n,\C)$. It turns out that the task of finding generators is actually easier if we do not impose any reality condition at all: in Section \ref{sec:GLnC}, we show that any $\GL(n,\C)$-valued loop can be written as a product of loops of the form
\[
p_{\al,\beta,V,W}(\la)=\left(\frac{\la-\al}{\la-\beta}\right)\pi_V +\pi_W,
\]
where the projections $\pi_V$ and $\pi_W$ are defined via a decomposition $\C^n=V\oplus W$ into complex subspaces, and
\[
m_{\al,k,N}(\la)=\Id+\left(\frac{1}{\la-\al}\right)^k N,
\]
where $k$ is a positive integer and $N$ is a two-step nilpotent map, i.e.~$N^2=0$.

Whereas the first type of simple elements is the obvious generalization of those used in \cite{Uhlenbeck1989}, the loops of the form $m_{\al,k,N}$ are of a different nature, mainly because they have only one singularity. This also reflects itself in the proof, which is split into two parts. Using only the first type and with the same arguments as in the proofs of the theorems mentioned above, we first reduce to the case of a loop with only one singularity; afterwards, a different argument shows that this loop is a product of loops of the second type.

In Section \ref{sec:GLnR}  we give a refinement of this proof to generate the subgroup of loops satisfying the reality condition given by $\GL(n,\R)$. We need those of the loops above that satisfy the reality condition, as well as products of two simple elements of the type $p_{\al,\beta,V,W}$ that do not satisfy the reality condition by themselves. 

We would like to remark that as previously done in the literature, we formulate the theorems for groups of negative loops, i.e.~loops that are normalized at $\infty$. All of them are true without this assumption, if we allow more general linear fractional transformations in the definition of the simple factors than those that send $\infty$ to $\Id$.

Sections \ref{sec:dressing1}, \ref{sec:dressing2} and \ref{sec:systems} are independent of the generating theorems in Sections  \ref{sec:GLnC} and \ref{sec:GLnR}.  In Section \ref{sec:dressing1} we consider the dressing action of simple elements of the form $m_{\al,k,N}$ with $k=1$, and apply this to the ZS-AKNS flows. To apply dressing to the twisted flows in the $\SL(n)/\SO(n)$--hierarchy, we also prove a permutability formula that enables us to find certain products $s_{\al,N}$ of simple elements $m_{\al,1,N}$ that satisfy the twisting condition, see \eqref{eq:salphaN}. In Section \ref{sec:dressing2} we briefly consider the case $k=2$.

Finally, in Section \ref{sec:systems}, we make the observation that the $n$--dimensional system associated to a symmetric space $U/K$ is equivalent to the system associated to its dual symmetric space $U^*/K$. The space of solutions of the $\Un(n)/\Or(n)$--system, which by the work of Terng and Wang \cite{Terng2006} can be identified with the space of $\partial$--invariant flat Egoroff metrics, is therefore acted on by the group of negative loops in $\GL(n,\C)$ satisfying the $\GL(n)$--reality and the $\GL(n)/\Or(n)$--twisting condition, in particular by the $s_{\al,N}$. We calculate the action of the $s_{\al,N}$ on those Egoroff metrics and their associated families of flat Lagrangian immersions in $\C^n$.

\section{Preliminaries} For any complex reductive Lie group $G$ and representation $\rho:G\to\GL(V)$, the \emph{rational loop group} associated to $\rho$ is given by
\[
\rlg(G,V)=\{g:\cp{1}\to G\mid \rho\circ g:\cp{1}\to \End(V) \text{ is meromorphic}\};
\]
see \cite{DFG2008} for some basic examples on how the rational loop group of $G$ depends on the chosen representation. 
If $\tau$ is an antiholomorphic involution of $G$, we say that a loop $g\in \rlg(\GL(n,\C))$ satisfies the \emph{reality condition} with respect to $\tau$ if
\[
\tau(g(\la))=g(\bar{\la}).
\]
If $\sigma$ additionally is an holomorphic involution on $G$ commuting with $\tau$, then we say that $g$ is {\emph{twisted}} with respect to $\sigma$ if 
\[\sigma(g(-\lambda))=g(\lambda).\]
A loop $g$ is called $\emph{negative}$ if it is normalized at $\infty$, i.e.~$g(\infty)=\Id$. We use superscripts to denote the reality and twisting conditions, and subscripts to denote negativity; for example, the group of negative rational loops satisfying the $\tau$-reality and the $\sigma$-twisting condition will be denoted by $\rlg_-^{\tau,\sigma}(G,V)$. 

If $g\in \rlg(G,V)$ is given, we say that $\alpha\in \cp{1}$ is a \emph{pole} of $g$ if $\al$ is a pole of $\rho\circ g:\cp{1}\to \End(V)$. If $\al$ is not a pole of $g$, we say that $\al$ is a \emph{zero} of $g$ if $\rho(g(\al))\in\End(V)$ is singular. Finally, $\al$ is a \emph{singularity} of $g$ if it is a pole or a zero. 

If $\al\in \cp{1}$ is a pole of $g$, there is a unique number $k\ge 1$ such that the map $(\la-\al)^{k-1}g$ has a pole at $\al$, but $(\la-\al)^k g$ has none. If we denote the evaluation of this map at $\al$ by $A\in \End(V)$, we call the pair $(k,\rm{rk }\, A)$ the \emph{pole data} of $g$ at $\al$. There is a natural ordering on the possible pole data: $(k_1,n_1)<(k_2,n_2)$ if and only if $k_1<k_2$ or ($k_1=k_2$ and $n_1<n_2$). It thus makes sense to compare degrees of poles.

\section{The full rational loop group} \label{sec:GLnC}
In this section we prove a Generating Theorem for the full rational loop group of $\GL(n,\C)$ associated to the standard representation on $\C^n$. The simple elements needed for that are given in Table \ref{tab:GLnC}. 

\begin{table}[h]
\caption{Simple elements for $\rlg_-(\GL(n,\C),\C^n)$}
\centering
\begin{tabular}{|c|c|c|}
\hline
Name & Definition & Conditions\\
\hline
$p_{\al,\beta,V,W}$ & $\left(\frac{\la-\al}{\la-\beta}\right)\pi_V +\pi_W$ & $\C^n=V\oplus W$ \\
\hline
$m_{\al,k,N}$ & $\Id+\left(\frac{1}{\la-\al}\right)^k N$ & $N:\C^n\to \C^n,\, N^2=0$\\
\hline
\end{tabular}
\label{tab:GLnC}
\end{table}
Here $\alpha$ and $\beta$ are distinct complex numbers. and the maps $\pi$ are projections along the decomposition in the column 'Conditions' onto the subspace in the subscript. Note that the $p_{\al,\beta,V,W}$ have two singularities, whereas the $m_{\al,k,N}$ have only one; furthermore, the determinant of $m_{\al,k,N}$ is $1$ at each value $\la\neq\al$.

\begin{thm}\label{thm:gennoreality}
The rational loop group $\rlg_-(\GL(n,\C),\C^n)$ is generated by the simple elements given in Table \ref{tab:GLnC}.
\end{thm}
\begin{rem} In the case $n=1$, no simple factors of the form $m_{\al,k,N}$ exist. The theorem becomes the well-known statement that any meromorphic map $f:\cp{1}\to \cp{1}$ with $f(\infty)=1$ is of the form $f(\la)=\frac{p(\la)}{q(\la)}$, where $p$ and $q$ are monic polynomials of equal degree.
\end{rem}
\begin{proof} Let $g\in \rlg_-(\GL(n,\C),\C^n)$. The first step in the proof is to multiply simple elements $p_{\al,\beta,V,W}$ to the left of $g$ to remove all but at most one singularity. This works similarly to the proofs of existing generating theorems:

Assume first that $g$ has at least two singularities. Let $\alpha\in \C$ be a pole of $g$ -- which exists since otherwise $g$ had to be constant -- and $\beta\in \C$ another singularity. If we define $\varphi(\la)=\frac{\la-\al}{\la-\beta}$, the map $g\circ \varphi^{-1}$ has a pole at $0$, so we can write its Laurent expansion around $0$ as $g\circ\varphi^{-1}(\la)=\sum_{j=-k}^\infty \la^j g_j$ with $g_{-k}\neq 0$. Composing with $\varphi$, we obtain the Laurent expansion of $g$ in $\frac{\la-\al}{\la-\beta}$ around $\al$:
\[
g(\la)=\sum_{j=-k}^\infty \left(\frac{\la-\al}{\la-\beta}\right)^j g_j.
\]
Let $V=\im g_{-k}$, choose an arbitrary complement $W$ of $V$, and regard
\begin{align*}
p_{\al,\beta,V,W}(\la)g(\la)&=\left( \left(\frac{\la-\alpha}{\la-\beta}\right) \pi_V +\pi_W\right)\left(\left(\frac{\la-\beta}{\la-\al}\right)^kg_{-k}+\ldots\right),
\end{align*}
which obviously has a pole at $\al$ of lower degree. Inductively, we can remove the pole at $\al$ by multiplying simple elements of the first type and are left with a loop (which we again call $g$) whose Laurent expansion around $\al$ in $\frac{\la-\al}{\la-\beta}$ is of the form
\[
g(\la)=g_0+\left(\frac{\la-\al}{\la-\beta}\right)g_1+\ldots
\]
If $g_0$ is invertible, we have completely removed the singularity $\al$. If $g_0$ is not invertible, we continue as follows: The map $\la\mapsto \det g(\la)$ has a zero at $\al$ of a certain order, say $l$. If we set $W=\im g_0$ and let $V$ be an arbitrary complement, the loop $
\tilde{g}=p_{\beta,\al,V,W}g$ has no pole at $\al$, and the order of the zero of $\la\mapsto \det \tilde{g}(\la)$ is lower than $l$. Using induction, we arrive at a loop whose evaluation at $\al$ is invertible, i.e., in $\GL(n,\C)$. This loop has strictly less singularities than the one we started with.

For this procedure, it was essential to be able to choose two distinct singularities. Therefore we can only repeat this process until we are left with a loop $g$ that has exactly one pole, say $\al\in \C$, and no other singularity, i.e., $g(\la)\in \GL(n,\C)$ for all $\la\in \cp{1},\la\neq \al$.
We can therefore write $g$ explicitly as
\[
g(\la)=(\la-\al)^{-r} A_r+\ldots+(\la-\al)^{-1} A_1+A_0
\]
with $A_r\neq 0$. The normalization condition says $A_0=\Id$. Since $\det g(\la)$ is a polynomial in $(\la-\al)^{-1}$, and complex nonconstant polynomials always have at least one pole and one zero on $\cp{1}$, it follows that $\det g(\la)=1$ for all $\la\neq \al$.

\vspace{0.4cm}
For the second part of the proof, we need some notation. For any $i\ge 0$, we define $K_i=\bigcap_{j\ge i} \ker A_j$ and 
\[
V_i:=\sum_{j\ge i} A_j(K_{j+1}).
\]
We have filtrations
\begin{equation}\label{eqn:filtration}
0=K_0\subset K_1\subset \ldots \subset K_r\subset K_{r+1}=\C^n
\end{equation}
and 
\[
0=V_{r+1}\subset V_r\subset \ldots \subset V_1\subset V_0\subset \C^n.
\]
Let ${\mathcal K}$ be the set of tuples of nonnegative integers $(a_i)_{i\ge 0}$ satisfying $\sum_i a_i=n$. We introduce a total ordering on ${\mathcal K}$ by setting
\[
(a_i)_i< (b_i)_i\Longleftrightarrow \text{There exists } j\ge 0 \text{ such that } a_i= b_i \text{ for } i>j \text{ and } a_j < b_j.
\]
Note that the unique minimum with respect to this ordering of $\mathcal K$ is the tuple $(n,0,0,\ldots)$.
For a loop $g$ as above, we define an associated tuple $\eps(g)=(a_i)_i\in {\mathcal K}$ by $a_i:=\dim K_{i+1}-\dim K_i=\dim A_i(K_{i+1})$; the tuple $\eps(g)$ really is an element of $\mathcal K$ since $\sum_{i\ge 0} (\dim K_{i+1}-\dim K_i)=\dim K_{r+1}-\dim K_0=n$. The only loop $g$ whose associated tuple $\eps(g)$ is the minimum $(n,0,0,\ldots)$ is the constant loop $g(\la)=\Id$. We show by induction on $\eps(g)$ that $g$ can be written as a product of simple elements of the form $m_{\al,k,N}$, the induction basis being trivial.

Let $s\ge 0$ be the smallest number such that $\im A_i\subset V_i$ for all $i\ge s$. Since $\im A_r=V_r$ by definition, we have $s\le r$. 

Let us first regard the case that $s>0$; the case $s=0$ will be treated later. By definition of $V_{s-1}$, the space $A_{s-1}(K_s)$ is a subset of $V_{s-1}$, but by definition of $s$, the space $A_{s-1}(K_{r+1})=\im A_{s-1}$ is not, so the smallest number $l$ such that $\im A_{s-1}(K_l)\not\subset V_{s-1}$ satisfies $s<l\le r+1$. Let $v\in K_l$ be such that $A_{s-1}(v)\notin V_{s-1}$, and note that $A_{l-1}(v)\neq 0$.

Let $N$ be a two-step nilpotent map satisfying $N(V_{s-1})=0$ and $N(A_{s-1}(v))=-A_{l-1}(v)\in V_{l-1}\subset V_{s-1}$. It follows that $NA_i=0$ for all $i\ge s$ since $\im A_i\subset V_i\subset V_{s-1}$ for such $i$. Therefore, the product
\begin{align*}
\tilde{g}(\la)&=m_{\al,l-s,N}(\la)g(\la)= (\Id+(\la-\al)^{s-l}N)\sum_{i=0}^r (\la-\al)^{-i}A_i\\
&=\Id+\ldots + (\la-\al)^{-l+1} (NA_{s-1}+A_{l-1}) +\sum_{i\ge l} (\la-\al)^{-i} A_i
\end{align*}
coincides with $g$ starting with the $(\la-\al)^{-l}$-coefficient. The $(\la-\al)^{-l+1}$-coefficient satisfies
\[
(NA_{s-1}+A_{l-1})(K_{l-1})=NA_{s-1}(K_{l-1})\subset N(V_{s-1})=0
\]
and 
\[
(NA_{s-1}+A_{l-1})(v)=-A_{l-1}(v)+A_{l-1}(v)=0,
\]
so $\eps(\tilde{g})<\eps(g)$, and induction may be applied.

It remains to regard the case $s=0$, i.e.~$\im A_i\subset V_i$ for all $i\ge 0$. In particular, $V_0=\C^n$. For dimensional reasons, we have a direct decomposition
\begin{equation}\label{eqn:basicproof_s=0}
\C^n=\bigoplus_{i\ge 0} A_i(K_{i+1}).
\end{equation}
Let $\mathcal B$ be a basis of $\C^n$ compatible with the filtration \eqref{eqn:filtration}. More precisely, let $W_i$ be a complement of $K_{i}$ in $K_{i+1}$, i.e.
\[
\C^n=K_r\oplus W_r=K_{r-1}\oplus W_{r-1}\oplus W_r=\ldots=\bigoplus_{i\ge 0} W_i, 
\]
choose bases ${\mathcal B}_i$ of $W_i$, and let ${\mathcal B}=\bigcup_i {\mathcal B}_i$. Note that $A_i(K_{i+1})=A_i(W_i)$, so by \eqref{eqn:basicproof_s=0}, we get a second basis $\mathcal B'$ of $\C^n$ by defining ${\mathcal B}'=\bigcup_i A_i({\mathcal B}_i)$. We have that ${\mathcal B}_j$ is in the kernel of $A_i$ whenever $j<i$, that $A_i$ sends ${\mathcal B}_i$ to ${\mathcal B}'_i$, and that $\im A_i \subset V_i$. Thus the matrix representation of $A_i$ with respect to these bases (${\mathcal B}$ as basis of the domain of definition, and ${\mathcal B'}$ as basis of the target) is of the form
\[
\left(\begin{array}{c|c|c} {\bf{0}} & {\bf{0}} & {\bf{0}} \\ \hline
{\bf{0}} & {\bf{1}} & *\\ \hline
\smash{\underbrace{\bf{0}}_{\dim K_i }} & \smash{\underbrace{\bf{0}}_{\dim K_{i+1}-\dim K_i}} & \smash{\underbrace{*}_{n-\dim K_{i+1}}}  \end{array} \right),
\]\vskip12pt
\noindent where $*$ signifies unknown entries and $\bf 1$ represents a diagonal matrix of the appropriate dimension. 

From this, we can calculate the leading term of $\det g(\la)$ as a polynomial in $(\la-\al)^{-1}$:
\[
\det g(\la)=(\la-\al)^{-\sum_{i\ge 0} i \dim W_i} +\ldots
\]
On the other hand, we know that $\det g(\la)=1$ for all $\la\neq \al$, which is therefore only possible if $\dim W_i=0$ for all $i\ge 1$, i.e.~$g(\la)=\Id$ for all $\la$.
\end{proof}

\section{The $\GL(n,\R)$--reality condition}\label{sec:GLnR}
In this section we prove a generating theorem for the group of $\GL(n,\C)$-valued loops satisfying the reality condition with respect to the noncompact real form $\GL(n,\R)$. Denote by $\tau:\GL(n,\C)\to \GL(n,\C)$ the antiholomorphic involution $\tau(A)=\bar{A}$; we are interested in the loop group $\rlg_-^\tau(\GL(n,\C),\C^n)$, i.e., the group of rational loops $g:\cp{1}\to \GL(n,\C)$ satisfying $\bar{g(\bar{\la})}=g(\la)$ and the normalization condition $g(\infty)=\Id$.

To generate this group, we need several types of simple elements, see Table \ref{tab:glnr}.
\begin{table}[h]
\caption{Simple elements for the $\GL(n,\R)$-reality condition}
\begin{tabular}{|c|c|c|}
\hline
Name & Definition & Conditions\\
\hline
$p_{\al,\beta,V,W}$ & $\left(\frac{\la-\al}{\la-\beta}\right)\pi_V +\pi_W$ & \begin{tabular}{c} $\alpha,\beta\in \R,\, \C^n=V\oplus W$ \\ $\bar{V}=V,\, \bar{W}=W$ \end{tabular} \\
\hline
$q_{\alpha,\beta,V,W}$ & $\frac{(\la-\alpha)(\la-\bar{\al})}{(\la-\beta)(\la-\bar{\beta})}\pi_V+\pi_W$ & \begin{tabular}{c} $\al$ or $\beta\notin \R,\, \C^n=V\oplus W$\\ $\bar{V}=V,\, \bar{W}=W$\end{tabular}\\
\hline
$r_{\alpha,\beta,V,W}$ & $\left(\frac{\la-\alpha}{\la-\beta}\right)\pi_V + \pi_W + \left(\frac{\la-\bar{\alpha}}{\la-\bar{\beta}}\right) \pi_{\bar{V}}$ & \begin{tabular}{c} $\C^n=V\oplus W\oplus \bar{V}$ \\ $V\cap \bar{V}=0,\, \bar{W}=W$ \end{tabular} \\
\hline
$m_{\al,k,N}$ & $\Id+\left(\frac{1}{\la-\al}\right)^k N$ & $\al\in\R,\, N^2=0,\, \bar{N}=N$\\
\hline
\end{tabular}
\label{tab:glnr}
\end{table}

Note that all of these simple elements are either $\GL(n,\C)$-simple factors or products of two $\GL(n,\C)$-simple factors that do not satisfy the reality condition by themselves: $q_{\alpha,\beta,V,W}=p_{\al,\beta,V,W}p_{\bar{\al},\bar{\beta},V,W}$ and $r_{\alpha,\beta,V,W}=p_{\al,\beta,V,W}p_{\bar{\al},\bar{\beta},\bar{V},W}$.

\begin{thm} \label{thm:glnr} The rational loop group $\rlg_-^\tau(\GL(n,\C),\C^n)$ is generated by the simple elements given in Table \ref{tab:glnr}.
\end{thm}

\begin{proof}Let $g\in \rlg_-^\tau(\GL(n,\C),\C^n)$. Observe that if $\alpha\in \C$ is a singularity of $g$, then so is $\bar{\al}$. We first regard the case that $g$ has at least two singularities, not all of which are real. Let $\al\in \C\setminus \R$ be a singularity of $g$. If $\alpha$ and $\bar{\alpha}$ are the only singularities of $g$, let $\beta$ be a random real number; otherwise let $\beta$ be a (real or complex) singularity of $g$ different from $\alpha$ and $\bar{\al}$. We will remove the singularity at $\alpha$ (and simultaneously at $\bar{\al}$) by multiplying with simple elements of the type $q$ and $r$, so although in the first case we might introduce a new singularity at the real value $\beta$, we will have reduced the total number of singularities in any case.

If $g$ has a pole at $\al$, write the Laurent expansion of $g$ in $\frac{\la-\al}{\la-\beta}$ around $\al$ as 
\[
g(\la)=\sum_{j=-k}^\infty \left(\frac{\la-\al}{\la-\beta}\right)^jg_j
\]
with $g_{-k}\neq 0$; otherwise continue with \eqref{eq:gonlyzero} below. 
If there exists a nonzero space $V\subset \im g_{-k}$ with $V=\bar{V}$, let $W$ be an arbitrary complement of $V$ in $\C^n$ with $\bar{W}=W$, and regard
\begin{align*}
q_{\al,\beta,V,W}(\la)g(\la)&= \left(\frac{(\la-\alpha)(\la-\bar{\al})}{(\la-\beta)(\la-\bar{\beta})}\pi_V + \pi_W \right)\left(\left(\frac{\la-\beta}{\la-\al}\right)^kg_{-k}+\ldots\right)\\
&=\left(\frac{\la-\beta}{\la-\al}\right)^k \pi_W\circ g_{-k}+\ldots.
\end{align*}
This loop has a pole of lower degree at $\al$, since the kernel of $\pi_W\circ g_{-k}$ contains not only the kernel of $g_{-k}$, but also the preimage of $V$ under $g_{-k}$.

If such a space does not exist, let $V=\im g_{-k}$ be the full image of $g_{-k}$. We have $V\cap \bar{V}=0$ and can therefore choose an arbitrary complement $W$ of $V\oplus \bar{V}$ with $\bar{W}=W$. Regard
\[
r_{\al,\beta,V,W}(\la)g(\la)=\left(\left(\frac{\la-\alpha}{\la-\beta}\right)\pi_V + \pi_W + \left(\frac{\la-\bar{\alpha}}{\la-\bar{\beta}}\right) \pi_{\bar{V}}\right)\left(\left(\frac{\la-\beta}{\la-\al}\right)^kg_{-k}+\ldots\right),
\]
which has a pole of lower degree; in fact, its $\left(\frac{\la-\beta}{\la-\al}\right)^k$-coefficient vanishes completely.

Continuing this, we obtain a loop (again denoted by $g$) without pole at $\al$, whose Laurent expansion in $\frac{\la-\al}{\la-\beta}$ around $\al$ we write as
\begin{equation}\label{eq:gonlyzero}
g(\la)=g_0+\left(\frac{\la-\al}{\la-\beta}\right)g_1+\ldots
\end{equation}
If $g_0$ is invertible, $\al$ is no singularity, so assume that $g_0$ is singular. Denote by $k$ the order of the zero $\al$ of the map $\la\mapsto \det g(\la)$.
Let $W_0\subset \im g_{0}$ be a maximal subspace with $W_0=\bar{W_0}$, and write $\im g_{-k}=W_0\oplus W_1$, where $W_1$ is an arbitrary complement of $W_0$ in $\im g_{-k}$. We have necessarily $W_1\cap \bar{W_1}=0$.

If $W_1$ is not empty, let $V=\bar{W_1}$ and $W=W_0\oplus W_2$, where $W_2$ is an arbitrary complement of $W_0\oplus W_1\oplus \bar{W_1}$ in $\C^n$ with $W_2=\bar{W_2}$. We have constructed a decomposition
\[
\C^n=V\oplus W\oplus \bar{V}
\]
with $\im g_{-k}\subset W\oplus \bar{V}$.
Then, the loop $\tilde{g}=r_{\beta,\al,V,W}g$ has no pole at $\al$ since $\al\neq \bar{\al}$; furthermore, the map $\la\mapsto \det \tilde{g}(\la)$ has a zero at $\al$ of lower order than $k$.

If $W_1$ is empty, we have $\im g_{-k}=W_0$, i.e.~$\im g_{-k}=\bar{\im g_{-k}}$. In this case, let $W=\im g_{-k}$ and $V$ be an arbitrary complement with $V=\bar{V}$. Then we reduce the order of the zero by regarding $\tilde{g}=q_{\beta,\al,V,W}g$.

By induction, we have removed the singularity $\alpha$ (and simultaneously $\bar{\alpha}$). Repeating this step removes all nonreal singularities.

\vspace{0.4cm}
After having removed all nonreal singularities, we have to deal with the case of several real singularities. If $\al\neq \beta$ are two real singularities of $g$, we can continue as in the first step, the difference being that the reality condition implies that the image of $g_{-k}$ (and the image of $g_0$, after having removed the pole) is invariant under conjugation. This simplifies matters insofar as we only need to make use of the simple factors $p$; in the notation of the previous step, there always exists a nonzero $V\subset \im g_{-k}$ with $\bar{V}=V$ (in fact, we may choose $V=\im g_{-k}$), and the space $W_1$ is always empty.

\vspace{0.4cm}
Finally, we are left with a loop $g\in \rlg_-^\tau(\GL(n,\C),\C^n)$ with exactly one singularity $\alpha\in \R$. We can therefore write $g$ explicitly as
\[
g(\la)=(\la-\al)^{-r} A_r+\ldots+(\la-\al)^{-1} A_1+A_0
\]
with $A_r\neq 0$. The reality condition implies immediately that $\bar{A_i}=A_i$ for all $i$. We may continue the proof exactly as in Theorem \ref{thm:gennoreality}, because due to the reality of the $A_i$, the nilpotent endomorphisms $N$ constructed there may all be chosen to be real.
\end{proof}

\section{Nilpotent dressing: Simple poles}\label{sec:dressing1}
Recall how the Birkhoff factorization theorem yields the \emph{dressing action} \cite{Terng2000} of the negative loop group $\rlg_-(\GL(n,\C))$ on the positive loop group $\rlg_+(\GL(n,\C))$: Given generic $g_\pm\in \rlg_\pm(\GL(n,\C))$, there exist $\hat{g}_\pm\in \rlg_\pm(\GL(n,\C))$ such that $g_-g_+=\hat{g}_+ \hat{g}_-$; the dressing action of $g_-$ on $g_+$ is then defined by $g_- * g_+:= \hat{g}_+$. Under presence of a $\tau$--reality and/or a $\sigma$--twisting condition, the dressing action restricts correspondingly (e.g.~we obtain an action of $\rlg_-^{\tau,\sigma}(\GL(n,\C))$ on $\rlg_+^{\tau,\sigma}(\GL(n,\C))$). Let us consider the dressing action of a nilpotent simple element
\[
m_{\alpha,1,N}=\Id + \left(\frac{1}{\la-\al}\right) N,
\] 
where $N^2=0$.

\begin{prop}\label{prop:nilpotentdressing}
Let $f\in \rlg_+(GL(n,\C))$, i.e.~$f:\C\to GL(n,\C)$ is holomorphic on all of $\C$. Let $f_1:=\left.\frac{d}{d\la}\right|_{\la=\al} f(\la)f(\al)^{-1}\in \gl(n,\C)$, and assume that $\Id+Nf_1$ is invertible. If we define $\tilde{N}:=f(\al)^{-1}(\Id+Nf_1)^{-1}Nf(\al)$, then $\tilde{N}^2=0$ and
\[
m_{\al,1,N} * f = m_{\al,1,N}f m_{\al,1,\tilde{N}}^{-1}\in \rlg_+(GL(n,\C)).
\]
\end{prop}
\begin{proof} To prove that $\tilde{N}$ is two-step nilpotent, multiply its defining equation 
\begin{equation}\label{eq:defineNtilde}
(\Id+Nf_1)f(\alpha)\tilde{N}f(\alpha)^{-1}=N
\end{equation}
from the left with $N$ to obtain
\begin{equation}\label{eq:NNtilde}
Nf(\alpha)\tilde{N}f(\alpha)^{-1}=0.
\end{equation}
Then, multiplying \eqref{eq:defineNtilde} from the right with $f(\alpha)\tilde{N}f(\alpha)^{-1}$, we get
\[
(\Id+Nf_1)f(\alpha)\tilde{N}^2f(\alpha)^{-1}=0,
\]
which is only possible if $\tilde{N}^2=0$.

To show holomorphicity, we only need to show that the loop is holomorphic at $\alpha$, i.e.~that the negative terms in its Laurent series expansion at $\al$ vanish. But the $(\la-\al)^{-2}$--coefficient is
\[
-Nf(\alpha)\tilde{N}=0
\]
using \eqref{eq:NNtilde}, and the $(\la-\al)^{-1}$--coefficient is
\[
Nf(\al)-f(\al)\tilde{N}-Nf_1f(\al)\tilde{N} = Nf(\al) -(\Id+Nf_1)(\Id+Nf_1)^{-1}Nf(\al)=0.
\]
Thus, the new loop is holomorphic.
\end{proof}
To give some first application of this proposition, let us quickly review the construction of the ZS-AKNS flows, developed by Zakharov and Shabat \cite{ZS} and Ablowitz, Kaup, Newell and Segur \cite{AKNS}. See e.g.~Section 2 of \cite{Terng2000} for a detailed exposition. For a non--zero diagonal matrix $a\in \lsl(n,\C)$, define
\[ {\lsl}(n,\C)_a=\{y\in {\lsl(n,\C)}\mid [a,y]=0\},\quad {\lsl}(n,\C)_a^\perp=\{y\in {\lsl(n,\C)}\mid {\tr}(ay)=0\},
\] and denote by $S(\R,\lsl(n,\C)_a^\perp)$ the space of rapidly decaying maps. For $b\in \lsl(n,\C)$ such that $[a,b]=0$ and any positive integer $j$, there is a unique family of $\lsl(n,\C)$--valued maps $Q_{b,j}$ such that 
\[
(Q_{b,j}(u))_x+[u,Q_{b,j}(u)]=[Q_{b,j+1}(u),a]
\]
and the asymptotic expansion $\sum_{j=0}^\infty Q_{b,j}(u)\la^{-j}$ is conjugate to $b$. Then, the $(b,j)$\emph{--flow} on $S(\R,\lsl(n,\C)_a^\perp)$, also called the \emph{$j$--th flow in the $\lsl(n,\C)$--hierarchy defined by $b$}, is given by
\[
u_t=(Q_{b,j}(u))_x+[u,Q_{b,j}(u)].
\]
If $u$ is a solution of the $j$--th flow defined by $b$, then there exists a unique \emph{trivialization} of $u$, i.e.~a solution $E(x,t,\la)$ of
\begin{align*} E^{-1}E_x&=a\la + u \\
E^{-1}E_t&= b\la^j+Q_{b,1}(u)\la^{j-1} + \ldots + Q_{b,j}(u) \\
E(0,0,\la)&=\Id.
\end{align*}
Assume that $u$ is a solution admiting a \emph{local reduced wave function} $\omega(x,t,\la)$, as in Definition 2.4 of \cite{Terng2000}. In particular,
\[
E(x,t,\la)=\omega(0,0,\la)^{-1}e^{a\la x + b\la^j t} \omega(x,t,\la).
\]
Then we can adapt Theorem 4.3 of \cite{Terng2000} to our situation:

\begin{prop} \label{prop:dressing} Let $u$ be a local solution of the $j$--th flow defined by $b$ with trivialization $E$ that admits a local reduced wave function $\omega$. Choose $\alpha\in \C$ and a two-step nilpotent map $N:\C^n\to \C^n$. Let $E_1(x,t)=\left.\frac{d}{d\la} \right|_{\la=\al} E(x,t,\la)E(x,t,\al)^{-1}$, and define $\tilde{N}$ as in Proposition \ref{prop:nilpotentdressing}:
\[
\tilde{N}(x,t)=E(x,t,\al)^{-1}(\Id+NE_1(x,t))^{-1}NE(x,t,\al),
\]
wherever this is well-defined. Then, $\tilde{u}(x,t)=u(x,t)-[a,\tilde{N}(x,t)]$ is another solution of the $j$--th flow. Its trivialization is
\[
\tilde{E}(x,t)=m_{\al,1,N} E(x,t) m_{\al,1,\tilde{N}(x,t)}^{-1}
\]
and it has the local reduced wave function
\[
\tilde{\omega}(x,t,\la)=\omega(x,t,\la)m_{\al,1,\tilde{N}(x,t)}(\la)^{-1}.
\]
\end{prop}
\begin{proof} The proof is as in \cite{Terng2000}. Since $m$ is a local reduced wave function of $u$, we have
\[
E(x,t,\la)=\omega(0,0,\la)^{-1}e^{a\la x + b\la^j t} \omega(x,t,\la).
\] 
Thus, if we define $\tilde{E}$ and $\tilde{m}$ as in the proposition, we have
\begin{align*}
\tilde{E}(x,t,\la)&=m_{\al,1,N} (\la) E(x,t,\la) m_{\al,1,\tilde{N}(x,t)}(\la)^{-1}\\
&=m_{\al,1,N} (\la) \omega(0,0,\la)^{-1}e^{a\la x + b\la^j t} \omega(x,t,\la) m_{\al,1,\tilde{N}(x,t)}(\la)^{-1}\\
&=\tilde{\omega}(0,0,\la)^{-1} e^{a\la x + b\la^j t} \tilde{\omega}(x,t,\la).
\end{align*}
Therefore, Proposition 2.11 of \cite{Terng2000} shows that if 
\[
\tilde{\omega}(x,t,\la)=\Id+\tilde{\omega}_1(x,t)\la^{-1}+\tilde{\omega}_2(x,t)\la^{-2}+\ldots
\]
is the expansion of $\tilde{\omega}$ at $\infty$, then $\tilde{u}=[a,\tilde{\omega}_1]$ is a solution of the $j$--th flow with trivialization $\tilde{E}$ and local reduced wave function $\tilde{\omega}$. We have
\[
m_{\al,1,\tilde{N}(x,t)}(\la)^{-1}=\Id-\tilde{N}(x,t)(\la-\al)^{-1} = \Id  - \tilde{N}(x,t)\la^{-1}+\ldots,
\]
and hence $\tilde{\omega}_1=\omega_1-\tilde{N}$. Thus, $\tilde{u}=u-[a,\tilde{N}]$.
\end{proof} 

\begin{exam}\label{exam:jthflowsimplepole}
Let us apply Proposition \ref{prop:dressing} to the vacuum solution $u=0$ of the $j$--th flow in the $\lsl(2,\C)$--hierarchy defined by $a=\left(\begin{matrix} 1 & 0 \\ 0 & -1 \end{matrix}\right)$. Its trivialization $E$ is given by $E(x,t,\la)=e^{a(\la x + \la^j t)}$, and its local reduced wave function is $\omega(x,t,\la)=\Id$. If we denote
\[
\xi(x,t)=x+j\al^{j-1}t,
\]
then the power series expansion of $E(x,t,\la)E(x,t,\al)^{-1}$ in $\la=\al$ reads 
\[
E(x,t,\la)E(x,t,\al)^{-1}= \Id+a\xi(x,t)(\la-\al)+\ldots,
 \]
 hence
\begin{align*}
\tilde{N}(x,t)
&=\left(\begin{matrix}e^{-\al x-\al^j t} & 0 \\ 0 & e^{\al x + \al^j t}\end{matrix}\right)(\Id+Na\xi(x,t))^{-1}N\left(\begin{matrix}e^{\al x+\al^j t} & 0 \\ 0 & e^{-\al x - \al^j t}\end{matrix}\right).
\end{align*}
We write the nilpotent matrix $N$ in the form $N=\left(\begin{matrix}n_1 & n_2 \\�n_3 & -n_1 \end{matrix}\right)$, with $\det N=-n_1^2-n_2n_3=0$. 
A direct calculation shows that
\begin{align*}
\tilde{N}(x,t)&= \frac{1}{1+2n_1\xi(x,t)}\left(\begin{matrix}n_1 & n_2 e^{-2\al x -2\al^j t} \\ n_3 e^{2\al x + 2\al^j t} & -n_1 \end{matrix}\right),
\end{align*}
and hence
\begin{align*}
\tilde{u}(x,t)&=-[a,\tilde{N}(x,t)]= \frac{2}{1+2n_1\xi(x,t)} \left(\begin{matrix}0 & -n_2 e^{-2\al x -2\al^j t} \\ n_3 e^{2\al x + 2\al^j t} & 0 \end{matrix}\right).
\end{align*}
We see that the new solution $\tilde{u}$ is smooth on all of $\R^2$ if and only if $n_1=0$, i.e.~$N=\left(\begin{matrix} 0 & n_2 \\ 0 & 0 \end{matrix}\right)$ or $N=\left(\begin{matrix} 0 & 0 \\ n_3 & 0 \end{matrix}\right)$. If $n_1\neq 0$, then $\tilde{u}$ is singular along the line $x+j\al^{j-1}t=-\frac{1}{2n_1}$.
\end{exam}

Consider the involutions $\sigma$ and $\tau$ on $\lsl(n,\C)$, given by $\tau(A)=\overline{A}$ and $\sigma(A)=-A^t$. The Cartan decomposition of the symmetric space $\SL(n)/\SO(n)$ is the eigenspace decomposition of $\sigma$, restricted to $\lsl(n,\R)$: $\lsl(n,\R)=\so(n)\oplus \lp$. For odd positive integer $j$, the \emph{$j$--th flow in the $\SL(n)/\SO(n)$--hierarchy defined by $b$} is given by the restriction of the $j$--th flow in the $\SL(n,\C)$--hierarchy to $S(\R,\lsl(n,\R)_{a,\sigma}^\perp)$,
where $\lsl(n,\R)_{a,\sigma}^\perp=\so(n)\cap \lsl(n,\C)_a^\perp$.

To apply dressing to twisted hierarchies, we need to find products of simple elements that satisfy the twisting condition. For that, a permutability formula is essential:

\begin{prop}
Let $\alpha \neq \beta$ and $N,M$ satisfy $N^2=M^2=0$. If
\[
\hat{N}=\left(\Id+\left(\frac{1}{\al-\beta}\right)M\right)\left(\Id+\left(\frac{1}{\al-\beta}\right)^2NM\right)^{-1}N\left(\Id-\left(\frac{1}{\al-\beta}\right)M\right)
\]
and
\[
\hat{M}=\left(\Id+\left(\frac{1}{\beta-\al}\right)N\right)\left(\Id+\left(\frac{1}{\beta-\al}\right)^2MN\right)^{-1}M\left(\Id-\left(\frac{1}{\beta-\al}\right)N\right).
\]
are well-defined, then we have
\[
m_{\beta,1,\hat{M}}m_{\al,1,N}=m_{\al,1, \hat{N}}m_{\beta,1,M}.
\]
\end{prop}
\begin{proof} This follows from Proposition \ref{prop:dressing} as usual. 
\end{proof}

For $\al\in \C$ and a two-step nilpotent map $N$ such that 
\begin{equation}\label{eq:twistingN'}
N'=\left(\Id-\frac{1}{2\al}N\right)\left(\Id+\frac{1}{4\al^2} N^tN\right)^{-1}N^t\left(\Id+\frac{1}{2\al}N\right)
\end{equation}
is well-defined, let 
\begin{equation}\label{eq:salphaN}
s_{\al,N}:=m_{-\al,1,N'}m_{\al,1,N}.
\end{equation}

\begin{cor} \label{cor:simpletwistedelements}
For $\al\in \R$ and $N$ a two-step nilpotent map with $\overline{N}=N$ such that \eqref{eq:twistingN'} is well-defined, we have $s_{\al,N}\in \rlg_-^{\tau,\sigma}(\GL(n,\C))$.
\end{cor}

\begin{exam}
The third flow in the $\SL(2,\R)/\SO(2)$--hierarchy defined by $a=\left(\begin{matrix}1 & 0 \\ 0 & -1 \end{matrix}\right)$ is the modified KdV equation
\[
q_t=\frac14 (q_{xxx}+6q^2q_x),
\]
where $u=\left(\begin{matrix}0 & q \\ -q & 0 \end{matrix}\right)$, see \cite{Terng2000}, Example 3.12. Let $\al\in \R$ and $N=\left(\begin{matrix}n_1 & n_2 \\ n_3 & -n_1 \end{matrix}\right)$ with $\det N=0$. To perform dressing with $s_{\al,N}$ on the vacuum solution $u=0$, we need to apply Proposition \ref{prop:dressing} twice. Using notation and the calculations of Example \ref{exam:jthflowsimplepole}, one finds the new solution $\hat{q}$ as the upper right entry of $\hat{u}=\tilde{u}-[a,\tilde{N'}]$, where $\tilde{N'}$ is constructed as follows:
\[
\tilde{N'}(x,t)=\tilde{E}(x,t,-\al)^{-1}(\Id+N'\tilde{E}_1(x,t))^{-1}N'\tilde{E}(x,t,-\al),
\]
where
\[
\tilde{E}(x,t)=m_{\al,1,N}E(x,t)m_{\al,1,\tilde{N}(x,t)}^{-1}
\]
and $\tilde{E}_1(x,t)=\left.\frac{d}{d\la}\right|_{\la=-\al} \tilde{E}(x,t,\la)\tilde{E}(x,t,-\al)^{-1}$. With the help of a computer one finds
\[
\hat{q}=-\al e^{2\al x + 2 \al^3 t}\frac{(A(x,t)-8n_1) n_3e^{4\al x+4\al ^3t}+(A(x,t)+8n_1) n_2}{n_3^2 e^{8ax+8a^3t}+B(x,t) e^{4ax+4a^3t}+n_2^2},
\]
where
\[
A(x,t)=16n_1\al x+48n_1\al ^3t+8\al
\]
and
\[
B(x,t)=16\al ^2 n_1^2 x^2+96 \al^4 n_1^2 x t+16 \al^2 n_1 x+48 \al^4 n_1 t+144 \al^6 n_1^2 t^2+4\al^2+2 n_1^2.
\]

\end{exam}

\section{Nilpotent dressing: Higher pole order}\label{sec:dressing2}

Dressing with simple elements $m_{\al,k,N}$ with $k\geq 2$ is still possible, but the formulas become more and more complicated as $k$ grows. We only give an idea for $k=2$.

\begin{prop} \label{prop:dressingpoleorder2} Let $f\in \rlg_+(GL(n,\C))$, choose $\al\in \C$ and a two-step nilpotent map $N$, and write the power series expansion of $f$ in $\al$ as $f(\la)=\sum_{i=0}^\infty f_i (\la-\al)^i$. If $X=(Nf_3+f_1)(Nf_2+f_0)^{-1}$,
\[
M_1=(Nf_2+f_0-XNf_1)^{-1}(XNf_0-Nf_1)
\]
and
\[
M_2=-(Nf_2+f_0)^{-1}(Nf_1M_1+Nf_0)
\]
are well-defined, then the loop
\[
m_{\al,2,N} f \left(\Id+\left(\frac{1}{\la-\al}\right)M_1+\left(\frac{1}{\la-\al}\right)^2M_2\right)
\]
is holomorphic at $\al$.
\end{prop}
\begin{proof}
The principal part of the Laurent series in $\al$ of the new loop reads
\begin{align}
&\left(\frac{1}{\la-\al}\right)^4 Nf_0M_2 \label{eq:dressingterm4} \\
+&\left(\frac{1}{\la-\al}\right)^3[Nf_1M_2+Nf_0M_1] \label{eq:dressingterm3}\\
+&\left(\frac{1}{\la-\al}\right)^2[(Nf_2+f_0)M_2 +Nf_1M_1+Nf_0]  \label{eq:dressingterm2}\\
+&\left(\frac{1}{\la-\al}\right)^1[(Nf_3+f_1)M_2+(Nf_2+f_0)M_1+Nf_1]. \label{eq:dressingterm1}
\end{align}
If the terms \eqref{eq:dressingterm2} and \eqref{eq:dressingterm1} vanish, then also \eqref{eq:dressingterm4} and \eqref{eq:dressingterm3}, as one can see by multiplying them from the left with the two-step nilpotent map $N$. But \eqref{eq:dressingterm2} and \eqref{eq:dressingterm1} vanish if $M_1$ and $M_2$ are chosen as in the statement of the proposition. 
\end{proof}


\begin{exam} Already the formulas in Example \ref{exam:jthflowsimplepole} become very complicated if one replaces the simple element $m_{\al,1,N}$ by $m_{\al,2,N}$. Therefore, we restrict ourselves to the case of the third flow, and the simple element having its pole at $0$.
For $u=\left(\begin{matrix} 0 & q \\ r & 0 \end{matrix} \right)$, the third flow in the $\lsl(2,\C)$--hierarchy defined by $a=\left(\begin{matrix}1 & 0 \\ 0 & -1 \end{matrix}\right)$ is given by
\[
q_t=\frac14 (q_{xxx}-6qrq_x),\quad r_t=\frac14 (r_{xxx}-6qrr_x),
\]
see \cite{Terng2000}, Example 2.8.
Applying Proposition \ref{prop:dressingpoleorder2} to the vacuum solution $u=0$ and the simple element $m_{0,2,N}$, with $N=\left(\begin{matrix}n_1 & n_2 \\ n_3 & -n_1\end{matrix}\right)$ satisfying $\det N=0$, a direct calculation provides the solution
\[
\tilde{u}=\frac{4}{4n_1^2x^4-12n_1^2xt+3} \left(\begin{matrix} 0 & n_2 (2n_1x^3+3x+3n_1t) \\ -n_3(2n_1x^3-3x+3n_1t) & 0 \end{matrix}\right).
\]
\end{exam}

\section{The $n$--dimensional systems}\label{sec:systems}

Let $U/K$ be a rank $n$ symmetric space with Cartan decomposition $\un=\lk\oplus \lp$, and choose a maximal abelian subalgebra $\fa\subset \lp$ with basis $a_1,\ldots,a_n$. Recall that the $n$--dimensional system associated to $U/K$ is the following system of first order partial differential equations for $v:\R^n\to \fa^\perp\cap \lp$:
\[
[a_i,v_{x_j}]-[a_j,v_{x_i}]=[[a_i,v],[a_j,v]],
\]
which is independent of the choice of basis. 

Associated to any symmetric space $U/K$ is its dual symmetric space $U^*/K$, which has the Cartan decomposition $\un^*=\lk\oplus i\lp$. Choosing the maximal abelian subspace $i\fa\subset i\lp$ with basis $ia_1,\ldots,ia_n$, we see
\begin{lem}  $v:\R^n\to \fa^\perp \cap \lp$ is a solution of the $U/K$--system if and only if $-iv:\R^n\to (i\fa)^\perp \cap i\lp$ is a solution of the $U^*/K$--system.
\end{lem}
Therefore, the $U/K$--system and the $U^*/K$--system are the same, and we do not only have a dressing action of the rational loop group $\rlg_-^{\tau,\sigma}(U)$ on the space of solutions of the $U/K$--system, but also one of $\rlg_-^{\tau,\sigma}(U^*)$.
Furthermore, whatever geometric interpretation of the solutions of the particular $U/K$--system has been found, also applies to the $U^*/K$--system.
 
Let us apply this observation to the system associated to the symmetric space $\GL(n)/\Or(n)$, which we now have seen to be the same as the system associated to $\Un(n)/\Or(n)$. The Cartan decomposition of $\GL(n)/\Or(n)$ is $\gl(n)=\so(n)\oplus \lp$, where $\lp$ is the space of symmetric matrices. Let $a_i=e_{ii}$ be the standard basis of the Cartan subalgebra $\fa\subset \lp$ of diagonal matrices, i.e.~$a_i$ is the matrix with zeros everywhere except a $1$ at the $ii$--entry. Then, $\beta:\R^n\to \lp$ is a solution of the $\GL(n)/\Or(n)$--system if and only if
\begin{equation}\label{eq:GLnOnsystem}
\begin{cases} (\beta_{ij})_{x_k}=\beta_{ik}\beta_{kj} & i,j,k \text{ distinct} \\
(\beta_{ij})_{x_i}+(\beta_{ij})_{x_j}+\sum_k \beta_{ik}\beta_{kj}=0 & i\neq j,
\end{cases}
\end{equation}
see \cite{Terng2006}. On the other hand, $\beta$ is a solution of the $\GL(n)/\Or(n)$--system if and only if $\omega_\la = \sum_i (\la a_i + [a_i,\beta])dx_i$ is flat for all $\la$. In this case, there is a unique frame $E(x,\la)$ satisfying
\[
E^{-1}dE=\sum_i  (\la a_i + [a_i,\beta])dx_i,\quad E(0,\la)=\Id.
\]
This frame satisfies the $\GL(n,\R)$--reality and the $\Or(n)$--twisting condition:
\[
E(x,\bar{\la})=\overline{E(x,\la)},\qquad E(x,-\la)^t E(x,\la)=\Id.
\]
\begin{rem} \label{rem:GLnUn} Observe that $F(x,\la)=E(x,i\la)$ satisfies the $\Un(n)$--reality condition:
\[
F(x,\overline{\la})^*F(x,\la)=E(x,i\overline{\la})^*E(x,i\la) = E(x,-i\la)^tE(x,i\la)=\Id.
\]
This is not surprising as $F^{-1}dF=\sum_i (i\la a_i + [a_i,\beta])dx_i$, i.e.~$F$ is the frame of the solution $-i\beta$ of the $\Un(n)/\Or(n)$--system.
\end{rem}
Let $\al\in \R$ and $N$ be a two-step nilpotent map with $\overline{N}=N$ such that both 
\[
\tilde{N}(x)=E(x,\al)^{-1}(\Id+NE_1(x))^{-1}NE(x,\al)
\]
and
\[
\tilde{N'}(x)=\tilde{E}(x,-\al)^{-1}(\Id+N'\tilde{E}_1(x))^{-1}N\tilde{E}(x,-\al)
\]
are well-defined. Here, $E_1$ and $\tilde{E}_1$ are given by $E_1(x):=\left.\frac{d}{d\la}\right|_{\la=\al} E(x,\la)E(x,\al)^{-1}$ and $\tilde{E}_1(x):=\left.\frac{d}{d\la}\right|_{\la=-\al} \tilde{E}(x,\la) \tilde{E}(x,-\al)^{-1}$, and $N'$ is given by  \eqref{eq:twistingN'}. Now we may consider the dressing action of a simple element $s_{\al,N}$ on $E$ (see \eqref{eq:salphaN} for the definition of $s_{\al,N}$) and obtain the new frame
\[
\hat{E}=s_{\al,N}*E=m_{-\al,1,N'}m_{\al,1,N} E m_{\al,1,\tilde{N}}^{-1} m_{-\al,1,\tilde{N'}}^{-1}.
\]
The calculation of $\hat{E}^{-1}d\hat{E}$ is then implicit in Proposition \ref{prop:dressing}:
\[
\hat{E}^{-1}d\hat{E}=\sum_i (\la a_i + [a_i,\beta-\tilde{N}-\tilde{N'}])dx_i,
\]
We have proved:

\begin{prop}\label{prop:dressingGLnOn} Let $\beta$ be a solution of the $\GL(n)/\Or(n)$--system, and $E(x,\la)$ its frame. Then
\[
s_{\al,N}* \beta= \beta-(\tilde{N}+ \tilde{N'})_*
\]
is the solution of the $\GL(n)/\Or(n)$--system obtained by dressing with $s_{\al,N}$. Here, we denote by $(\tilde{N}+ \tilde{N'})_*$ the trace--free part of $\tilde{N}+ \tilde{N'}$.
\end{prop}

Let us quickly review parts of the connection between solutions of the $\Un(n)/\Or(n)$--system (resp.~the $\GL(n)/\Or(n)$--system) and Egoroff metrics, as found by Terng and Wang \cite{Terng2006}. A local orthogonal system $(x_i)$ of $\R^n$ is called Egoroff if there exists a function $\phi(x)$ such that the Euclidean metric $ds^2$ written in this coordinate system is of the form $ds^2=\sum_i h_i^2(x) dx_i^2$, 
where $h_i^2(x)=\frac{\partial \phi}{\partial x_i}$. The rotation coefficient matrix $\beta$ of the Egoroff metric $\sum h_i^2dx_i^2$ is defined by $\beta_{ij}=\frac{(h_i)_{x_j}}{h_j}$ for $i\neq j$, and $\beta_{ii}=0$. If $\beta$ is the rotation coefficient matrix of a flat Egoroff metric, then $\beta$ solves \eqref{eq:GLnOnsystem}, i.e.~is a solution of the $\GL(n)/\Or(n)$--system. Conversely, if $\beta$ is a solution of the $\GL(n)/\Or(n)$--system, then $\beta$ is the rotation coefficient matrix of a flat Egoroff metric.

A flat Egoroff metric is called $\partial$--invariant or spherical, if $\partial h_i=0$, where $\partial=\sum_j \frac{\partial}{\partial x_i}$ --- see Proposition 2.4 of \cite{Terng2006}, where four equivalent conditions for being $\partial$--invariant are listed. Recall also statements (1) and (3) of Theorem 2.5 of \cite{Terng2006}:  If $\sum_i h_i^2 dx_i^2$ is a $\partial$--invariant flat Egoroff metric, and $E$ the frame of $\sum_i (\la a_i + [a_i,\beta])dx_i$, then $h$ can be reconstructed via the formula $E(x,0)h(x)=h(0)$. Furthermore, there is an associated family of flat Lagrangian immersions into $\C^n$ given by
\[
X(x,\la)=-i\la^{-1}(E(x,i\la)h(x)-h(0)).
\]
Note that the additional factor $i$ in front of $\la$ is explained by Remark \ref{rem:GLnUn}. Then, we have the following analogue of Theorem 4.2 of \cite{Terng2006}:
\begin{prop} Let $\sum_i h_i^2dx_i^2$ be a $\partial$--invariant flat Egoroff metric with coefficient matrix $\beta$ and frame $E(x,\la)$. Let $c=h(0)$. If $\hat{E}=s_{\al,N} * E$ and $\hat{c}$ is a constant, then we have a new $\partial$--invariant flat Egoroff metric
\[
\hat{h}(x)=\hat{E}(x,0)c
\] 
with associated family of flat Lagrangian submanifolds
\[
\hat{X}(x,\la)=-i\la^{-1} (\hat{E}(x,i\la)\hat{E}(x,0)^{-1}\hat{c}-\hat{c}).
\]
\end{prop}

\newcommand{\noopsort}[1]{}

\end{document}